\documentclass[12pt,a4paper]{article}
\usepackage{amsmath,amsthm,amsfonts,latexsym,amscd,amssymb,mathrsfs,hyperref,textcomp}
\newtheorem{theorem}{Theorem}[section]
\newtheorem{lemma}[theorem]{Lemma}
\newtheorem{corollary}[theorem]{Corollary}
\newtheorem{proposition}[theorem]{Proposition}
\newtheorem{example}[theorem]{Example}

\newtheorem{remark}[theorem]{Remark}

\begin{document}

\title{On Transiso Graph}
%
\author{ Vipul Kakkar$^1$
~ and Laxmi Kant Mishra$^2$\footnote{Author is supported by University Grants Commission, Government of India.}\\
$^1$School of Mathematics, Harish-Chandra Research Institute\\ Allahabad, India\\
$^2$Department of Mathematics, University of Allahabad \\
Allahabad (India) 211002\\
Email: vplkakkar@gmail.com; lkmp02@gmail.com}
\date{}
\maketitle
\begin{abstract}
In this note, we define a new graph $\Gamma_d(G)$ on a finite group $G$, where $d$ is a divisor of $|G|$. The vertices of $\Gamma_d(G)$ are the subgroups of $G$ of order $d$ and two subgroups $H_1$ and $H_2$ of $G$ are said to be adjacent if there exists $S_i \in \mathcal{T}(G,H_i)$ $(i=1,2)$ such that $S_1 \cong S_2$. We shall discuss the completeness of $\Gamma_d(G)$ for various groups like finite abelian groups, dihedral groups and some finite $p$-groups. 
\end{abstract}
\section{Introduction}\label{int}

Let $G$ be a finite group and $H$ be a subgroup of $G$. A \emph{right transversal} $S$ of $H$ in $G$ is a subset of $G$ obtained by selecting one and only one element from each right coset of $H$ in $G$. $S$ is \emph{normalized right transversal (NRT)} if $1 \in S$. An NRT $S$ has an induced binary operation $\circ$ given by $\{x \circ y\}=S\cap Hxy$, with respect to which $S$ is a right loop with identity $1$, that is, a right quasigroup with both sided identity  (see \cite[Proposition 2.2, p.42]{smth},\cite{rltr}). Conversely, every right loop can be embedded as a normalized right transversal in a group with some universal property (see \cite[Theorem 3.4, p.76]{rltr}). Let $\langle S \rangle$ be the subgroup of $G$ generated by $S$ and $H_S$ be the subgroup $\langle S \rangle \cap H$. Then  $H_S=\langle \{xy(x \circ y)^{-1}|x, y \in S \} \rangle$ and $H_{S}S=\langle S \rangle$ (see \cite{rltr}).

Identifying $S$ with the set $H \backslash G$ of all right cosets of $H$ in $G$, we get a transitive permutation representation $\chi_{S}:G\rightarrow Sym(S)$ def{i}ned by $\{\chi_{S}(g)(x)\}=S\cap Hxg, ~ g\in G, x\in S$. The kernal $ker\chi_S$ of this action is $Core_{G}(H)$, the core of $H$ in $G$.
\\Let $G_{S}=\chi_{S}(H_{S})$. This group is known as the \textit{group torsion} of the right loop $S$ (see \cite[Definition 3.1, p.75]{rltr}). The group $G_S$ depends only on the right loop structure $\circ$ on $S$ and not on the subgroup $H$. Since $\chi_S$ is injective on $S$ and if we identify $S$ with $\chi_S(S)$, then $\chi_S(\langle S \rangle)=G_SS$ which also depends only on the right loop $S$ and $S$ is an NRT of $G_S$ in $G_SS$. One can also verify that $ ker(\chi_S|_{H_SS}: H_SS \rightarrow G_SS)=ker(\chi_S|_{H_S}: H_S \rightarrow G_S)=Core_{H_SS}(H_S)$ and $\chi_S|_S$=the identity map on $S$. If $H$ is a corefree subgroup of $G$, then there exists an NRT $T$ of $H$ in $G$ which generates $G$ (see \cite{cam}). In this case, $G=H_TT \cong G_TT$ and $H=H_T \cong G_T$. Also $(S, \circ)$ is a group if and only if $G_S$ trivial.

Let $\mathcal{T}(G, H)$ denote the set of all normalized right transversals (NRTs) of $H$ in $G$. Two NRTs $S$ , $T \in \mathcal{T}(G, H)$ are said to be \emph{isomorphic} (denoted by $S \cong T$), if their induced right loop structures are isomorphic.
A subgroup $H$ is normal in $G$ if and only if all NRTs of $H$ in $G$ are isomorphic to $G/H$ (see \cite[p.70]{rltr}).

Let $V$ be a set. Denote by $E(V)=\{\{u,v\}|u,v \in V, u \neq v\}$, the $2$-sets of $V$. A pair $\Gamma=(V,E)$ with $E \subseteq E(V)$ is called a graph on $V$ (see \cite{rd}). The elements of $V$ are called the \textit{vertices} of $\Gamma$ and those of $E$ the edges of $\Gamma$. If $\{u,v\} \in E$, then we say that $u$ and $v$ are adjacent. The number $|V|$ is called the order of $\Gamma$. A graph of order $0$ or $1$ is called as the trivial graph. The graph $(\emptyset,\emptyset)$ is called as the empty graph. The graph $\Gamma$ is called as  \textit{complete} if any two vertices are adjacent. 

\section{Transiso Graph}\label{tg}

Let $G$ be a finite group and $d$ be divisor of $|G|$ (order of $G$).Let $V_d$ be the set of all subgroups of $G$ of order $d$. We define a graph $\Gamma_d(G)=(V_d,E_d)$ with $\{H_1,H_2\} \in E_d$ if and only if there exists $S_i \in \mathcal{T}(G,H_i)$ $(i=1,2)$ such that $S_1 \cong S_2$ with respect to the right loop structure induced on $S_i$. We will call this graph a \textit{transiso graph}.


\begin{example}\label{e1} 
\begin{enumerate}
	\item \label{ex1}A finite cyclic group $C_n$ of order $n$ has unique subgroup corresponding to each divisor $d$ of $n$ so $\Gamma_d(C_n)$ is pointed graph for each divisor $d$ of $n$.
	\item \label{ex2}Let $G=C_2\times C_4=\langle a,b;a^2,b^4,aba^{-1}b^{-1}\rangle$. One can easily observe that $V_2=\{H_1=\langle a\rangle,H_2=\langle b^2\rangle,H_3=\langle ab^2\rangle\}$ and $V_4=\{K_1=\langle b\rangle,K_2=\langle ab\rangle,H_3=\langle a,b^2\rangle\}$. Also $H_1\cong H_2\cong H_3\cong C_2$, $G/H_1\cong G/H_3\cong C_4$, $G/H_2\cong C_2\times C_2$, $K_1\cong K_2\cong C_4$, $K_3\cong C_2\times C_2$, $G/K_1\cong G/K_2\cong G/K_3 \cong C_2$. We show the connectivity of subgroups in following pictorial form:
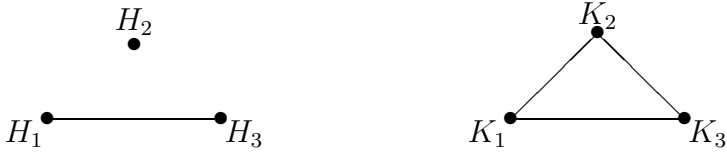
\setlength{\unitlength}{0.3in}
\begin{picture}(0,0)(2.7,2.5)
\put(0,0){\line(1,0){3}}
\put(-0.4,-.25){\makebox(0,0){$H_1$}}
\put(3.4,-.25){\makebox(0,0){$H_3$}}
\put(1.5,1.7){\makebox(0,0){$H_2$}}
\put(1.5,1.3){\makebox(0,0){$\bullet$}}
\put(0,0){\makebox(0,0){$\bullet$}}
\put(3,0){\makebox(0,0){$\bullet$}}
\put(8,0){\line(1,0){3}}
\put(8,0){\line(1,1){1.5}}
\put(11,0){\line(-1,1){1.5}}
\put(7.6,-.25){\makebox(0,0){$K_1$}}
\put(11.4,-.25){\makebox(0,0){$K_3$}}
\put(9.5,1.8){\makebox(0,0){$K_2$}}
\put(11,0){\makebox(0,0){$\bullet$}}
\put(8,0){\makebox(0,0){$\bullet$}}
\put(9.5,1.5){\makebox(0,0){$\bullet$}}
\end{picture}
\end{enumerate}
\end{example}
\vspace{2.3cm}

\begin{proposition} \label{p1}
A subgroup of a group $G$ is always adjacent with its automorphic images in $\Gamma_d(G)$ for any divisor $d$ of $|G|$.
\end{proposition}

\begin{proof}
Let $H$ be a subgroup of a group $G$. Let $f$ be an automorphism of $G$ and $K=f(H)$. Choose $S_1\in \mathcal{T}(G,H)$. Let $S_2=f(S_1)$. Observe that $S_2\in \mathcal{T}(G,K)$.

Let $x,y \in S_1$. Then $\{f(x\circ y)\}=f\{x\circ y\}=f(S_1\cap Hxy)=f(S_1)\cap f(Hxy)=S_2\cap Kf(x)f(y)=\{f(x)\circ f(y)\}$. This implies that $f|_{S_1}:S_1\rightarrow S_2$ is a right loop isomorphism. Hence, $H$ and $K$ are adjacent in $\Gamma_d(G)$.
\end{proof} 
\noindent Converse of Proposition \ref{p1} is not true in general (see Example \ref{e1}). Let $\delta(G)$ denotes the number of divisors of $|G|$ for which there is a subgroup of $G$ of that order.


\begin{corollary}
If number of orbits of the action of $\operatorname{Aut}(G)$ on the set $V$ of all subgroups of $G$ is equal to $\delta(G)$, then $\Gamma_d(G)$ is complete for each divisor $d$ of $|G|$.
\end{corollary}


\begin{proposition}\label{p2}
Let $H_1$ and $H_2$ be corefree subgroups of $G$. Let $S_i \in \mathcal{T}(G,H_i)$ ($i=1,2$) such that $S_1 \cong S_2$ and $\langle S_i \rangle=G$.
Then an isomorphism between $S_1$ and $S_2$ can be extended to an automorphism of $G$ which sends $H_1$ onto $H_2$.
\end{proposition}

\begin{proof}
Note that $H_{S_i}=\langle S_i\rangle\cap H_i=G\cap H_i=H_i$ ($i=1,2$) and hence $H_{S_i}S_i=G$.
Let $p:S_1\rightarrow S_2$ be a right loop isomorphism. It gives rise to an isomorphism $\tilde{p} :G_{S_1}S_1\rightarrow G_{S_2}S_2$ such that $\tilde{p}(G_{S_1})=G_{S_2}$ (see the discussion following \cite[Lemma 2.5, p. 2684]{rpsc}).
Since $H_i$ ($i=1,2$) is corefree, $\chi_{S_i}$ defined in the section $1$ is an isomorphism. Now, one can note $\chi_{S_2}^{-1} \circ \tilde{p} \circ \chi_{S_1}$ is an automorphism of $G$ which is an extension of $p$ and sends $H_1$ to $H_2$.
\end{proof}


Let $H$ be a subgroup of $Sym(n)$ of order $2$ and $S\in \mathcal{T}(Sym(n),H)$. Then $H_S=\langle S\rangle\cap H$ has order at most $2$. Assume that $|H_S|=2$. Then $H_S=H$, so $|\langle S\rangle|=|HS|=n!$ and hence $\langle S\rangle=Sym(n)$. Assume that $|H_S|=1$. Then $S$ is a subgroup of  $Sym(n)$ of order $\frac{n!}{2}$. One can note that $H$ can not be generated by an even permutation of order $2$. In this case, $ S =Alt(n)$, where $Alt(m)$ denotes the alternating group of degree $m$. 


\begin{example}
\begin{enumerate}
\item Let $H$ and $K$ be distinct subgroups of $Sym(n)$ of order $2$ generated by odd permutations. Then one can observe that $Alt(n) \in \mathcal{T}(Sym(n),H)\cap \mathcal{T}(Sym(n),K)$. Therefore $H$ and $K$ are adjacent in $\Gamma_2(Sym(n))$.

\item  Let $H$ and $K$ are subgroups of $Sym(n)$ of order $2$ generated by even permutations. Then as argued in the above paragraph, all NRTs of $H$ and $K$ generate $Sym(n)$. If $H$ and $K$ are adjacent, then by Proposition \ref{p2} there is an automorphism sending $H$ onto $K$. If $n\neq 6$, then $H$ and $K$ conjugate (for $\operatorname{Aut}(Sym(n))\cong \operatorname{Inn}(Sym(n))$ for $n\neq 6$ (see \cite[p.300]{suz})). If $n=6$, then $H$ and $K$ are again conjugate (for all automorphisms of $Sym(6)$ send a permutation of cycle type $(2,2,1,1)$ to permutation of cycle type $(2,2,1,1)$).
Moreover, if $H$ and $K$ are conjugate, the by Proposition \ref{p1} $H$ and $K$ are adjacent.

\item Now if $H$ is generated by an even permutation of order $2$ and $K$ is generated by an odd permutation of order $2$, then they are not adjacent in $\Gamma_2(Sym(n))$ (for otherwise they will be conjugate).

\item Since all subgroups of $Alt(n)$ ($n \leq 5$) of same order are conjugate, by Proposition \ref{p1} $\Gamma_d(Alt(n))$ ($n \leq 5$) is complete for each divisor $d$ of $\frac{n!}{2}$. As argued above, one can observe that $\Gamma_2(Alt(n))$ ($n \geq 8$)  is not complete.

\item Let $n \geq 6$. Let $H_1$ and $H_2$ be two subgroups of $Alt(n)$ such that $H_1\cong C_4$ and $H_2\cong C_2\times C_2$. Since $Alt(m)$ ($m\geq 5$) has no subgroup of index less than $m$, $|H_{S_i}| \neq 1$ and $|H_{S_i}| \neq 2$ ($i=1,2$). This implies that $|H_{S_i}|=4$ ($i=1,2$). Hence, all members of $\mathcal{T}(Alt(n),H_i)$ generates $Alt(n)$. Now, by Proposition \ref{p2} one observes that $H_1$ and $H_2$ are not adjacent in $\Gamma_4(Alt(n))$.    
\end{enumerate}
\end{example}

\begin{proposition}\label{pc}
For a finite abelian group $G$, $\Gamma_d(G)$ is complete for each divisor $d$ of $|G|$ if and only if each sylow subgroup of $G$ is either elementary abelian or cyclic.
\end{proposition}

\begin{proof}
One can easily check the 'if' part. We will observe the 'only if' part.

Assume that $G$ is not isomorphic to the group stated in the proposition and $\Gamma_d(G)$ is complete for each divisor $d$ of $|G|$. Then by Fundamental Theorem of abelian groups, there must be a summand of $G$ isomorphic to $C_{p^\alpha}\times C_{p^\beta}$ for some prime divisor $p$ of $|G|$, where $\alpha,\beta$ are positive integers and atleast one of $\alpha$ and $\beta$ is greater than $1$. Without any loss, let us assume that $\alpha >1$ and $C_{p^\alpha}\times C_{p^\beta}$ is at first place.

If $\alpha>\beta\geq 1$, then there exist two subgroups $H_1$ and $H_2$  such that $H_1 \cong C_p\times \{1\}\times \cdots$ and $H_2 \cong \{1\}\times C_p \times \cdots$. One can observe that $G/H_1\ncong G/H_2$. This is a contradiction.

If $\alpha=\beta>1$, then there exist two subgroups $H_1$ and $H_2$ such that $H_1 \cong C_{p^2}\times \{1\}\times\cdots$ and $H_2 \cong C_p\times C_p \times\cdots$.  Then $G/H_1\ncong G/H_2$. This is again a contradiction.
\end{proof}

\begin{corollary}
If $G$ is elementary abelian group, then $\Gamma_d(G)$ is complete for each divisor $d$ of $|G|$.
\end{corollary}


Let $D_{2n}$ denotes the dihedral group of order $2n$. We need following elementary lemma to prove that $\Gamma_d(D_{2n})$ is complete for each divisor $d$ of $|G|$. The proof of following Lemma can be found in \cite[Theorem 2.37, p. 54]{sr} and \cite[Theorem 3.3, p. 5]{kc}. 
\begin{lemma} \label{l1}
A subgroup a dihedral group $D_{2n}=\langle a,b;a^n,b^2,(ba)^2\rangle$ is either cyclic or dihedral. Moreover if $m$ is a divisor of $2n$ and
\begin{enumerate}
	\item $m$ is odd then all $m$ subgroups of index $m$ are conjugate to $\langle a^m,b\rangle$.
	\item $m$ is even and $m$ does not divide $n$ then there is only one subgroup $\langle a^{\frac{m}{2}}\rangle$ of index $m$.
	\item $m$ is even and $m$ divides $n$ then a subgroup of index $m$ is either $\langle a^{\frac{m}{2}}\rangle$ or conjugate to exactly one of $\langle a^m,b\rangle$ or $\langle a^m,ba\rangle$.	
\end{enumerate}
\end{lemma}

\begin{proposition}
Let $D_{2n}$ denote the dihedral group of order $2n$. Then $\Gamma_d(D_{2n})$ is complete for each divisor $d$ of $2n$.
\end{proposition}
\begin{proof}
Let $G=D_{2n}=\langle a,b;a^n,b^2,(ba)^2\rangle$. Let $d$ be a divisor of $2n$ and $m=\frac{2n}{d}$. 

Assume that $m$ is odd. Then by Lemma \ref{l1}, there are $m$ subgroups of $G$ of order $d$ and all are conjugate to $\langle a^m, b\rangle$. Therefore, by  Proposition \ref{p1}, $\Gamma_d(G)$ is complete for the divisor $d$ for which $m$ is odd.

Assume that $m$ is even and $m$ does not divide $n$. Then, by Lemma \ref{l1} there is only one subgroup $\langle a^{\frac{m}{2}}\rangle$ of order $d$. Therefore, $\Gamma_d(G)$ is complete for the divisor $d$ for which $m$ is even and $m$ does not divide $n$.

Finally, assume that $m$ is even and $m$ divides $n$. Then by Lemma \ref{l1}, a subgroup of order $d$ is either $\langle a^{\frac{m}{2}}\rangle$ or conjugate to exactly one of $\langle a^m,b\rangle$ or $\langle a^m,ba\rangle$. Let $H_1=\langle a^{\frac{m}{2}}\rangle$, $H_2=\langle a^m,b\rangle$ and $H_3=\langle a^m,ba\rangle$.
Note that $H_1$ is a normal subgroup of $G$.  Hence, all NRTs of $H_1$ in $G$ are isomorphic to $G/H_1 \cong D_{2 \cdot\frac{m}{2}}$.

Now choose $S_2=\{1,a^2,a^4,\ldots,a^{m-2},ba,ba^3,ba^5,\ldots,ba^{m-1}\}=\{a^{2i},ba^{2j+1}|0\leq i,j\leq (\frac{m}{2}-1)\}$ in $\mathcal{T}(D_{2n},H_2)$ and $S_3=\{a^{2i},ba^{2j}|0\leq i,j\leq (\frac{m}{2}-1)\}$ in $\mathcal{T}(D_{2n},H_3)$. Note that $\langle S_2\rangle =\langle a^2,ba \rangle$ and $\langle S_3\rangle =\langle a^2,b \rangle$. Then
$H_{S_2}=\langle S_2 \rangle \cap H_2=\langle a^m \rangle \trianglelefteq \langle S_2 \rangle$ and $H_{S_3}=\langle S_3 \rangle \cap H_3=\langle a^m \rangle \trianglelefteq \langle S_3 \rangle$. Therefore $G_{S_2}=G_{S_3}=\{1\}$ and hence $S_2$ and $S_3$ are groups.

Let $\circ_2$ denote the induced binary operation on $S_2$ as described in the first paragraph of section \ref{int}. One can observe that, $(a^2)^{\frac{m}{2}}=(ba)^2=(ba\circ_2a^2)^2=1$. This implies that $S_2\cong D_{2\cdot \frac{m}{2}}$. Similarly $S_3 \cong D_{2\cdot \frac{m}{2}}$. This shows that $\Gamma_d(G)$ is complete also in this case.
\end{proof}
Let $Q_8$ denote the quaternion group of order $8$. Then each subgroup of $Q_8$ is normal in $Q_8$. Hence, $\Gamma_d(Q_8)$ is complete graph for each divisor $d$ of $8$. Note that each NRT of subgroups of order $2$ generates $Q_8$. The converse of this is also true as observed below: 
\begin{proposition}
Let $G$ be a non-abelian finite $p$-group with all NRTs of a subgroup of index greater than $p$ generate $G$ and $\Gamma_d(G)$ is complete for each divisor $d$ of $|G|$. Then $G\cong Q_8$.
\end{proposition}

\begin{proof}
 Let $H_1$ and $H_2$ be subgroups of $G$ of same index greater than $p$. Since each finite $p$-group has a normal subgroup for each divisor of $|G|$, we can take $H_2\trianglelefteq G$. Choose $S_1\in \mathcal{T}(G,H_1)$, $S_2\in \mathcal{T}(G,H_2)$ such that $S_1 \cong S_2$. By the assumption $\langle S_1\rangle=\langle S_2\rangle=G$. This implies that $H_{S_1}=H_1$, $H_{S_2}=H_2$. Since $H_2 \trianglelefteq G$, $S_2$ is a group. This implies that $G_{S_2}=\{1\}$. Since $S_1 \cong S_2$, $G_{S_1} \cong G_{S_2}$ (see the discussion following \cite[Lemma 2.5, p. 2684]{rpsc}).  Since $G_{S_1}\cong H_{S_1}/Core_G{H_{S_1}}=H_1/Core_G(H_1)$, $H_1$ is also normal in $G$. Hence $G$ is a Dedekind group (see \cite[p.143]{rob}). Since $G$ is finite $p$-group, by \cite[Theorem 5.3.7, p. 143]{rob} $G$ is isomorphic to $Q_8$ or $Q_8 \times A$, where $A$ is an elementary abelian $2$-group. Assume that $G \cong Q_8 \times A$. Let $H$ be a subgroup of $A$ of order $2$. Then one can note that an NRT of $H$ in $G$ does not generate $G$. This is a contradiction. Thus $G \cong Q_8$ 
\end{proof}

\section{Complete Transiso Graph For Lower Prime Power Order Groups}
In this section, we determine the completeness of the transiso graph for finite $p$-group $G$ for divisor $p$ ($p$ an odd prime) upto order $p^5$. We will show that if $|G|=p^4$, then transiso graph $\Gamma_p(G)$ is not complete. For the group of order $p^5$, $\Gamma_p(G)$ is not complete except $\Phi(G)=G^{\prime}=Z(G) \cong C_p \times C_p$, where $\Phi(G)$, $G^{\prime}$, $Z(G)$ and $C_p$ denotes the Frattini subgroup, commutator subgroup, center of $G$ and cyclic group of order $p$ respectively. Using the small group library of GAP (\cite{gap}), we found that the transiso graph $\Gamma_3(G)$ for the $37^{th}$ group of order $3^5$ is complete. This group is of exponent $3$. We will observe that for the extra special group $G$ of order $p^3$ of exponent $p$, the transiso graph $\Gamma_d(G)$ is complete for each divisor $d$ of $p^3$. 

We further ask to determine the structure of the group $G$ of order $p^n$ for which transiso graph $\Gamma_d(G)$ is complete for each divisor $d$ of $p^n$. This problem can be thought as a dual problem posed in \cite{ra}. In \cite{ra}, R. Armstrong proves only finite non-abelian $p$-group all of whose subgroups of same order are isomorphic is the group of order $p^3$ of exponent $p$. In case of a complete transiso graph we feel that the only finite non-abelian $p$-group for which transiso graph $\Gamma_d(G)$ is complete for each order $d$ is the group of order $p^3$ of exponent $p$. 

Throughout the section, we will adopt following convention. The prime $p$ will always be odd. The group $G$ denotes  finite $p$-group which is not $p$-central, that is $G$ has non-central subgroup of order $p$. Whenever we write $H$ or $H_i$ ($i \in \mathbb{N}$), we will always mean that this is a non-normal subgroup of $G$ of order $p$. Whenever we write the semidirect product $G_1 \ltimes G_2$ of the groups $G_1$ and $G_2$, we will mean that it is not a direct product.

\begin{proposition}\label{3sp1}  
Let $G$ be a non $p$-central finite $p$-group. Then $\Gamma_p(G)$ is complete if and only if whenever $H$ is a non-normal subgroup of $G$ of order $p$, $G \cong H \ltimes K$ for some subgroup $K$ of $G$ with $G/L \cong K$ for any normal subgroup $L$ of $G$ of order $p$.
\end{proposition}
\begin{proof}
One can easily observe the 'if' part. We will only prove 'only if' part. 
Let $H \ntrianglelefteq G$ of order $p$. Let $L$ be any normal subgroup of $G$ of order $p$. Since $\Gamma_p(G)$ is complete, there exists $S_1 \in \mathcal{T}(G,H)$ and $S_2 \in \mathcal{T}(G,L)$ such that $S_1 \cong S_2$. This implies that $G_{S_1} \cong G_{S_2}$ (see the discussion following \cite[Lemma 2.5, p. 2684]{rpsc}). Since $L \trianglelefteq G$, $G_{S_2}=\{1\}$. Also, since $H$ is core-free subgroup of $G$, $H_{S_1} \cong G_{S_1}$. This implies that $H_{S_1} \cong G_{S_1} \cong G_{S_2}=\{1\}$. This means that $S_1$ is subgroup of $G$. Denote it by $K$. Note that $K \trianglelefteq G$. This implies that $G \cong H \ltimes K$ and $K \cong G/L$ for any $L \trianglelefteq G$.     
\end{proof}

\begin{corollary}\label{3se1} Let $G$ be a non $p$-central finite $p$-group with $|\Phi(G)|=p$ and $\Gamma_p(G)$ is complete. Then $G \cong H \ltimes K$, where $K \cong C_p \times \cdots \times C_p$ ($p-1$ times).
\end{corollary}  

\begin{proposition}\label{3sl3}
Let $G$ be a finite $p$-group ($p$ odd prime) with the property that whenever $H$ is a non-normal subgroup of $G$ of order $p$, $G$ is the semidirect product of $H$ and a normal subgroup $K$ such that all subgroups of $K$ of order $p$ are normal in $G$ and $K$ is isomorphic to the quotient $G/L$ for any normal subgroup $L$ of $G$ of order $p$. Then $K$ is a cyclic group.
\end{proposition}
\begin{proof}
Let $H=\langle h \rangle$ be a non-normal subgroup of order $p$ and $G=H \ltimes K$. Since $K=G/L$ for any subgroup $L$ of $K$ of order $p$ and all subgroups of $K$ of order $p$ are normal in $G$, it follows that the image $HL/L$ of $H$ in $G/L$ is normal and hence also central in $G/L$. So $[h,g] \in L$ for all $g \in G$. Now if $K$ had another subgroup $L^{\prime}$ of order $p$, then we would also have $[h,g] \in L^{\prime}$ and hence $[h,g]=1$ so $H \leq Z(G)$, contrary to assumption. So $K$ has a unique subgroup $L$ of order $p$. This implies that $K$ is a cyclic group.
\end{proof}

\begin{example}\label{3se2} One can note that $\Gamma_p(G)$ is complete, if $|G| \leq p^2$. Let $G$ be a group of order $p^3$ and $G$ be non-abelian. From the classification of group of order $p^3$, we note that $|Z(G)|=|\Phi(G)|=p$. By the classification of groups of order $p^3$, there are two non-abelian groups upto isomorphism. One is of exponent $p$ and other is of exponent $p^2$. Let $G$ be a non-abelian group of order $p^3$ of exponent $p^2$. Assume that $\Gamma_p(G)$ is complete. By Corollary \ref{3se1}, $G \cong C_p \ltimes (C_p \times C_p)$. One can note that there is unique subgroup $K$ of $G$ isomorphic to $C_p \times C_p$. Let $H$ be non-normal subgroup of $G$ of order $p$ contained in $K$. Then $G \ncong H \ltimes K$. Thus, by Proposition \ref{3sp1} $\Gamma_p(G)$ is not complete. This is a contradiction. Now, assume that $G$ is non-abelian group of order $p^3$ of exponent $p$. We will now prove that $\Gamma_d(G)$ is complete for each divisor $d$ of $p^3$.

First note that all subgroup of $G$ of order $p^2$ are normal in $G$ and their quotient are isomorphic to $C_p$. Hence $\Gamma_{p^2}(G)$ is complete. 

Let $H$ be any non normal subgroup of order $p$. Choose $x \in G \setminus (Z(G)H)$. Take $K=\langle x \rangle Z(G)$. One can note that $G \cong H \ltimes K$. By Proposition \ref{3sp1}, $\Gamma_p(G)$ is complete.
\end{example}

\begin{example}\label{3se3}
In Example \ref{3se2}, we have seen that if $G$ is of order $p^3$ of exponent $p$, then $\Gamma_d(G)$ is complete for each divisor $d$ of $p^3$. We now calculate number of vertices of trasiso graph. Note that vertices are subgroups of same order.

Since $G$ is of exponent $p$, number of elements of order $p$ is $p^3-1$. Also if $H$ is a subgroup of $G$ of order $p$, then there are $p-1$ elements in $H$ of order $p$. Thus, there are $\frac{p^3-1}{p-1}=p^2+p+1$ subgroups of order $p$.

Note that each subgroup of $G$ of order $p^2$ is isomorphic to $C_p \times C_p$. Let $x$ and $y$ be two elements of $G$ of order $p$ which generates a subgroup of order $p^2$. Note that $xy=yx$. This means that the order of the centralizer $C_G(x)$ of $x$ is atleast $p^2$. If $|C_G(x)|=p^3$, then $x \in Z(G)$. If $|C_G(x)|=p^2$, then $C_G(x)=\langle x,y \rangle$. This means that $Z(G) \subseteq \langle x,y \rangle$. Without any loss, we assume that $x \in Z(G)$. Therefore, to count subgroups of order $p^2$ we only take care of choices for $\langle y \rangle$ which are non central. There are $p^2+p$ choices for $\langle y \rangle$. Note that there are $p$ subgroups of $\langle x, y \rangle$ which are not central. Thus there are $\frac{p^2+p}{p}=p+1$ subgroups of order $p^2$.  
\end{example}

Let $G$ be a non-abelian group of order $p^4$. Then one can easily observe that $\Phi(G)$ is abelian. Let $G$ be a non-abelain group of order $p^5$. Then from a result of \cite{bkm}, which states that \textit{if $G$ is a finite $p$-group ($p$ odd prime) and the center $Z(\Phi(G))$ is cyclic, then $\Phi(G)$ is cyclic}, one can observe that $\Phi(G)$ is abelian in this case also. Now, we have following:

\begin{lemma}\label{3sl1}
Let $G$ be a non-abelian finite $p$-group of order either $p^4$ or $p^5$. Then $\Phi(G)$ is not a cyclic group unless $|\Phi(G)|=p$.
\end{lemma}
\begin{proof}
Let $|G|=p^4$. Assume that $\Phi(G) \cong C_{p^2}$. By \cite[Lemma 3.1, p. 304]{db}, $G \cong Q \times E$, where $E$ is elementary abelian
and both $\Phi(Q)$ and $Z(Q)$ are cyclic. Assume that $|E|=p^2$. Then $|Q|=p^2$. Hence $G$ is abelian. This is a contradiction. Similarly $|E|=p^3$ is not possible. Thus, $|E|=p$, $|Q|=p^3$ and $Q$ will be a non abelian group. Since $|Q|=p^3$, $|\Phi(Q)|=p$. But, by \cite[Theorem, p. 22]{gam}, $\Phi(G) \cong \Phi(Q)$. This is a contradiction.

Let $|G|=p^5$. If $\Phi(G)$ is cyclic group of order greater that $p$, either $\Phi(G) \cong C_{p^2}$ or $\Phi(G) \cong C_{p^3}$. Assume that $\Phi(G)=C_{p^2}$. As argued above, the cases $|E|=p^4$ and $|E|=p^3$ are not possible. Therefore, assume that $|E|=p^2$. Then $|Q|=p^3$ and $Q$ is non-abelian. Since $|\Phi(Q)|=p$ and $|\Phi(G)|=p^2$, by \cite[Theorem, p. 22]{gam} we get a contradiction. Assume that $|E|=p$. Then $|Q|=p^4$. By \cite[Theorem, p. 22]{gam}, $\Phi(G) \cong \Phi(Q) \cong C_{p^2}$. Since $|Q|=p^4$, we get a contradiction. Thus $\Phi(G) \ncong C_{p^2}$. Similarly $\Phi(G) \ncong C_{p^3}$.   
 \end{proof}

\begin{lemma}\label{3sl2}
Let $G$ be a finite $p$-group and $\Gamma_p(G)$ is complete. Then all subgroups of order $p$ which are contained in $\Phi(G)$ are normal in $G$.
\end{lemma}
\begin{proof}
Assume that $H$ be subgroup of $G$ contained in $\Phi(G)$ which is not normal in $G$. Then by Proposition \ref{3sp1}, $G \cong H \ltimes K$ for some subgroup $K$ of $G$. But by \cite[Lemma 2.11, p. 1916]{dx}, $H \trianglelefteq G$. This is a contradiction.
\end{proof}

\begin{corollary}\label{3sl2c}
Let $G$ be a non-abelain group of order either $p^4$ or $p^5$. Let $|Z(G)|=p$ and $\Gamma_p(G)$ is complete. Then $\Phi(G)=Z(G)$.
\end{corollary}

\begin{proposition}\label{3sp2}
Let $G$ be a finite $p$-group such that the order of the commutator subgroup $G^{\prime}$ is $p$. If $G^{\prime} \subsetneqq \Phi(G)$ and $\Phi(G)$ is not cyclic, then $\Gamma_p(G)$ is not complete. 
\end{proposition}
\begin{proof}
Assume that $\Gamma_p(G)$ is complete. By Lemma \ref{3sl2}, all subgroups of $\Phi(G)$ of order $p$ are normal in $G$. By Proposition \ref{3sp1}, $G \cong H \ltimes K$ and $K \cong G/L$ for any normal subgroup $L$ of order $p$. Take $L_1=G^{\prime}$ and $L_2 \subseteq \Phi(G)$ such that $L_2 \neq L_1$ and $|L_2|=p$. By Proposition \ref{3sp1}, $G/L_1 \cong G/L_2$. This implies that $G^{\prime} \subseteq L_2$. This is a contradiction.
\end{proof}

\begin{corollary}\label{3sp2c}
Let $G$ be a group of order $p^4$ such that $\Phi(G)=Z(G) \cong C_p \times C_p$. Then $\Gamma_p(G)$ is not complete.
\end{corollary}
\begin{proof}
Since $|\Phi(G)|=p^2$, $G=\langle x,y \rangle$ for some $x,y \in G$. Since $G^{\prime} \leq Z(G)$, $G^{\prime} = \langle [x, y] \rangle$. But since $Z(G)$ is an elementary abelian group, $|G^{\prime}|=p$. By Proposition \ref{3sp2}, $\Gamma_p(G)$ is not complete.
\end{proof}
 
We have following theorem for groups of order $p^4$:

\begin{theorem}\label{3st1}
Let $G$ be a group of order $p^4$. Then $\Gamma_p(G)$ is not complete.
\end{theorem}
\begin{proof}
Assume that $\Gamma_p(G)$ is complete. By Lemmas \ref{3sl1}, \ref{3sl2} and Proposition \ref{3sp2}, we are left with $|G^{\prime}|=|\Phi(G)|=p$. Let $H$ be a non-normal subgroup of $G$. By Proposition \ref{3sp1}, $G=H \ltimes K$ and $K \cong G/\Phi(G) \cong C_p \times C_p \times C_p$. By \cite[\textsection 3, p. 64]{hab}, we have only one choice for $G$ whose numbering in above cited reference is given by $1$. By \cite[\textsection 3, p. 64]{hab}, $G^{\prime} \cong C_p$ and $Z(G) \cong C_p \times C_p$. Since $|G^{\prime}|=p$, we can choose a subgroup $K$ of $G$ of order $p$ distinct from $G^{\prime}$. By Proposition \ref{3sp1}, $G/K \cong G/G^{\prime}$. This is a contradiction.
 \end{proof}

\begin{proposition}\label{3st2}
Let $G$ be a group of order $p^5$ such that $|\Phi(G)|=p$. Then $\Gamma_p(G)$ is not complete.
\end{proposition}
\begin{proof}
Assume that $\Gamma_p(G)$ is complete. Note that $G^{\prime}=\Phi(G)$. By Corollary \ref{3se1}, $G \cong H \ltimes K$ where $H$ is a non-normal subgroup of $G$ of order $p$ and $K \cong C_p \times C_p \times C_p \times C_p$. By \cite[\textsection 4, p. 65]{hab}, we have only one choice for $G$ which is numbered by $1$ in this reference. By \cite[\textsection 4, p. 65]{hab}, $Z(G) \cong C_p \times C_p \times C_p$. Choose $y \in Z(G) \setminus G^{\prime}$. Take $L_1=G^{\prime}$ and $L_2=\langle y \rangle$. Then one can observe that $G/L_1 \ncong G/L_2$. This is a contradiction to the Proposition \ref{3sp1}.   
\end{proof}

\begin{proposition}
Let $G$ be group of order $p^5$ such that $|\Phi(G)|=p^3$. Then $\Gamma_p(G)$ is not complete.
\end{proposition}
\begin{proof}
Assume that $\Gamma_p(G)$ is complete. By Lemma \ref{3sl1}, $\Phi(G) \cong C_p \times C_p \times C_p$ or $\Phi(G) \cong C_p \times C_{p^2}$. Assume that $\Phi(G) \cong C_p \times C_p \times C_p$. By Lemma \ref{3sl2}, $\Phi(G) \subseteq Z(G)$. Since $|G|=p^5$, $\Phi(G)=Z(G)$. By \cite{hab}, there is no group of order $p^5$ such that $|Z(G)|=p^3$ and $|G^{\prime}|=p^2$ and $|Z(G)|=p^3$ and $|G^{\prime}|=p^3$. Thus $|G^{\prime}|=p$. But by Proposition \ref{3sp2}, we get a contradiction. Thus $\Phi(G) \cong C_p \times C_{p^2}$.

By Lemma \ref{3sl2}, $Z(G)$ contains a subgroup isomorphic to $C_p \times C_p$. By \cite{hab}, there is no group of order $p^5$ such that $|Z(G)|=p^3$ and $G^{\prime} \cong C_p \times C_{p^2}$ or $|Z(G)|=p^3$ and $|G^{\prime}|=p^2$. Hence $Z(G) \cong C_p \times C_p$ and $|G^{\prime}|$ is either $p$ or $p^2$. By Proposition \ref{3sp2}, $|G^{\prime}| \neq p$. Hence $|G^{\prime}|=p^2$. Assume that $G^{\prime} \cong C_{p^2}$. By \cite[\textsection 4, p. 618]{rj}, there is a group in isoclinism family $\phi_8$. But in this case, $|Z(G)|=p$. This is a contradiction. Thus $G^{\prime} \cong C_p \times C_p$. By Lemma \ref{3sl2}, $Z(G)=G^{\prime}$. By \cite[\textsection 4, p. 618]{rj}, there are groups in isoclinism family $\phi_4$. But, in this case $G/Z(G)$ is an elementary abelian $p$-group. This is a contradiction.     
\end{proof}

\begin{proposition}
Let $G$ be group of order $p^5$ such that $|\Phi(G)|=p^2$. Then $\Gamma_p(G)$ is not complete unless $\Phi(G)=G^{\prime}=Z(G) \cong C_p \times C_p$.
\end{proposition}
\begin{proof}
Let $G$ be a group of order $p^5$ other than the case $\Phi(G)=G^{\prime}=Z(G) \cong C_p \times C_p$. Assume that $\Gamma_p(G)$ is complete. By Lemma \ref{3sl1}, $\Phi(G) \cong C_p \times C_p$. By Proposition \ref{3sp2}, $|G^{\prime}| \neq p$. Hence $\Phi(G)=G^{\prime}$. By Lemma \ref{3sl2}, $\Phi(G) \subseteq Z(G)$. By \cite{hab}, there is no group of order $p^5$ such that $|Z(G)|=p^3$ and $|G^{\prime}|=p^2$. Thus $\Phi(G)=G^{\prime}=Z(G) \cong C_p \times C_p$. This is a contradiction.
\end{proof}

\begin{remark}
We write SmallGroup($n$,$m$) for the $m^{th}$ group of order $n$ as quoted
in the "Small Groups" library in GAP (\cite{gap}). Using GAP calculations, we found that SmallGroup($243$,$37$) satisfies the condition of Proposition \ref{3sp1}. Hence $\Gamma_3(G)$ is complete, where $G$=SmallGroup($243$,$37$). One can check that $G \cong C_3 \ltimes (C_3 \times C_3 \times C_3 \times C_3)$ and $\Phi(G)=G^{\prime}=Z(G) \cong C_3 \times C_3$.
\end{remark} 

\noindent \textbf{Acknowledgement:} Authors are thankful to Dr. R. P. Shukla, Department of Mathematics, University of Allahabad, India for suggesting this problem and his valuable discussion. Authors are also thankful to Prof. Derek Holt, Mathematics Institute, University of Warwick, U.K. regarding Proposition \ref{3sl3}.  


\begin{thebibliography}{99}
\bibitem{ra}Ruth Armstrong
{\em Finite groups in which any two subgroups of the same order are isomorphic},
Math. Proc. Cambridge Philos. Soc., 54, (1958), 18-27.

\bibitem{abg}A. Ballester-Bollinches and G. Xiuyun
{\em On complemented subgroups of finite groups},
Arch. Math., 72, (1999), 161 - 166.

\bibitem{hab}H. A. Bender
{\em A Determination of the Groups of Order $p^5$},
Ann. of Math., Second Series, (1927 - 1928), 29(1/4), 61-72.

\bibitem{bkm}T. R. Berger; L. G. Kovacs and M. F. Newman
{\em Groups of prime power order with cyclic Frattini subgroup},
Nederl. Akad.Wetensch. Indag. Math., (1980), 42(1), 13-18.


\bibitem{db}D. Bornand
{\em Elementary abelian subgroups in p-groups with a cyclic
derived subgroup},
Journal of Algebra, (2011), 335, 301-318.

\bibitem{wb}W. Burnside,
{\em Theory of groups of finite order},
(1897), Cambridge University Press. 

\bibitem{cam} P. J. Cameron, preprint available at\\ {\it http://www.maths.qmul.ac.uk/\texttildelow pjc/preprints/transgenic.pdf}.

\bibitem{kc} K. Conrad, notes available at\\ {\it http://www.math.uconn.edu/\texttildelow kconrad/blurbs/grouptheory/dihedral2.pdf}.


\bibitem{dx}Li Deyu and Guo Xiuyun
{\em The influence of c-normality of subgroups on the structure of finite groups II},
Comm. Algebra, (1998), 26:6, 1913-1922.

\bibitem{rd} R. Diestel, {\it Graph Theory} (2005) (New York: Springer-Verlag).

\bibitem{gap} The GAP Group {\em GAP-Groups, Algorithms, and Programming, Version 4.4.10, 2007}, http://www.gap-system.org.

\bibitem{rj}R. James
{\em On the groups of order $p^6$ ($p$ an odd prime)},
Mathematics of Computation, (1980), 34(150), 613-637.

\bibitem{gam}G. A. Miller
{\em The $\phi$ subgroup of a group},
Tran. A.M.S., (1915), 16(1), 20-26.

\bibitem{rltr} R. Lal,  Transversals in Groups, {\it J. Algebra} \textbf{181} (1996) 70-81.
\bibitem{rob} D. J. S. Robinson, {\it A Course in the Theory of Groups} (1996) (New York: Springer-Verlag).

\bibitem{sr} S. Roman, {\it Fundamentals of Group Theory: An Advanced Approach} (2012) (New York: Birkhauser).

\bibitem{smth} J. D. H. Smith, {\it An Introduction to Quasigroups and Their Representations} (2007) (Boca Raton, FL: Chapman and Hall/CRC).
\bibitem{rpsc} R. P. Shukla,  Congruences in Right Quasigroups and General Extensions, {\it Comm. Algebra} \textbf{23(7)} (1995) 2679-2695.

\bibitem{suz} M. Suzuki, {\it Group Theory I} (1982) (New York: Springer-Verlag).
\end{thebibliography}
\end{document}